\documentclass[12pt]{amsart}

\usepackage{latexsym, amssymb, amscd, amsfonts}
\usepackage[all,cmtip]{xy}

\usepackage{bm}
\usepackage[colorlinks=true,urlcolor=black,citecolor=black,linkcolor=black,%
pdftitle={On Kostant root systems of Lie superalgebras},%
pdfauthor={Ivan Dimitrov and Rita Fioresi},%
pdfsubject={Representation theory},%
pdfkeywords={Parabolic subalgebras, Kostant root systems, Positive roots}]{hyperref}
\usepackage{ifthen}

\usepackage{hyperref}

\DeclareMathAlphabet{\mathpzc}{OT1}{pzc}{m}{it}

\newboolean{bourbaki}
\setboolean{bourbaki}{false}

\newcommand{\np}{\medskip\noindent}

\newcounter{eqcounter}[section]
\renewcommand{\theeqcounter}{\arabic{section}.\arabic{eqcounter}}
\renewenvironment{equation}{\medskip\noindent\refstepcounter{eqcounter}\makebox[0pt][l]{({\bf\theeqcounter})}\begin{minipage}[b]{\textwidth}$$}{$$\end{minipage}\medskip\noindent}

\renewcommand{\geq}{\geqslant}
\renewcommand{\leq}{\leqslant}

\newcommand{\D}{\mathrm{D}}                            
\newcommand{\E}{\mathrm{E}}                            
\newcommand{\F}{\mathrm{F}}                            
\newcommand{\G}{\mathrm{G}}                            

\newcommand{\cS}{\mathcal{S}}
\newcommand{\cZ}{\mathcal{Z}}
\newcommand{\cD}{\mathcal{D}}
\def\vep{\varepsilon}

\def\ga{\mathfrak{a}}

\def\gc{\mathfrak{c}}

\def\gg{\mathfrak{g}}

\def\gh{\mathfrak{h}}

\def\gk{\mathfrak{k}}
\def\gl{\mathfrak{l}}
\def\gm{\mathfrak{m}}

\def\gp{\mathfrak{p}}

\def\gs{\mathfrak{s}}
\def\gt{\mathfrak{t}}
\def\gu{\mathfrak{u}}

\def\gz{\mathfrak{z}}
\def\ggl{\mathfrak{gl}}
\def\gsl{\mathfrak{sl}}
\def\ggl{\mathfrak{gl}}
\def\gpsl{\mathfrak{psl}}
\def\gosp{\mathfrak{osp}}
\def\gso{\mathfrak{so}}
\def\gsp{\mathfrak{sp}}
\def\gsu{\mathfrak{su}}
\def\gu{\mathfrak{u}}

\def\e{{\bar{0}}}
\def\o{{\bar{1}}}

\newcommand{\ep}{\epsilon}                             

\ifthenelse{\boolean{bourbaki}}
{
\newcommand{\CC}{\mathbf{C}} 
\newcommand{\QQ}{\mathbf{Q}} 
\newcommand{\RR}{\mathbf{R}} 
\newcommand{\ZZ}{\mathbf{Z}} 
}
{
\newcommand{\CC}{\mathbb{C}} 
\newcommand{\QQ}{\mathbb{Q}} 
\newcommand{\RR}{\mathbb{R}} 
\newcommand{\ZZ}{\mathbb{Z}} 
}

\newcommand{\point}{\refstepcounter{subsection}\noindent{\bf \thesubsection.} }

\newtheorem{thm}{Theorem}[section]
\newtheorem{theorem}[thm]{Theorem}

\newtheorem{corollary}[thm]{Corollary}
\newtheorem{lemma}[thm]{Lemma}

\newtheorem{proposition}[thm]{Proposition}
\newcommand{\Deo}{\bar{\Delta}}
\newcommand{\Ro}{\bar{R}}

\theoremstyle{definition}
\newtheorem{definition}[thm]{Definition}

\newtheorem{remark}[thm]{Remark}

\newcommand{\cH}{\mathcal{H}}

\newcommand{\rk}{{\mathrm{rk\,}}}

\begin{document}
\pagestyle{plain} \title{{ \large{On Kostant root Systems for Lie superalgebras}}
}
\author{Ivan Dimitrov}
\address{Ivan Dimitrov: Department of Mathematics and Statistics, Queen's University, Kingston,
Ontario,  K7L 3N6, Canada} 
\email{dimitrov@queensu.ca} 
\thanks{Research of I.\ Dimitrov was partially supported by an NSERC grant}
\author{Rita Fioresi}
\address{Rita Fioresi: Departamento di Matematica, Universit\'a di Bologna, Piazza di Porta S. Donato 5, 40126 Bologna, Italy}
\email{rita.fioresi@unibo.it}


\begin{abstract} We study the eigenspace decomposition of a basic classical Lie superalgebra 
under the adjoint action of a toral subalgebra, thus extending results of Kostant. In recognition of Kostant's
contribution we refer to the eigenspaces appearing in the decomposition as Kostant roots. We then prove that
Kostant root systems inherit the main properties of classical root systems. Our approach is combinatorial
in nature and utilizes certain graphs naturally associated with Kostant root systems.  
In particular, we reprove Kostant's results without making use of the Killing form.

\np
\end{abstract}

\maketitle

\vspace{0.5cm}

\section{Introduction}

\np
\point {\bf Background.}
Let $\gg$ be a simple complex Lie algebra and let $\gt$ be a toral subalgebra of $\gg$, i.e., $\gt$ is an abelian subalgebra consisting of 
semisimple elements. Under the adjoint action of $\gt$, $\gg$ decomposes as

\begin{equation} \label{krs}
\gg=\gm \, \oplus \, \left(\oplus_{\nu \in R} \, \gg_\nu\right), 
\end{equation}
where 
\[\gm  = \gc_\gg(\gt) = \{x \in \gg \, | \, [t,x] = 0 {\text { for every }} t \in \gt\}\]
is a reductive subalgebra of $\gg$, 
\[\gg_\nu  = \{x \in \gg\, | \, [t,x]=\nu(t)x {\text { for every }} t \in \gt\}\] 
and
\[R = \{ \nu \in \gt^*\backslash\{0\} \, | \, \gg_\nu \neq 0\}.\] 
The elements of $R$ are called {\it Kostant roots} (or {\it $\gt$-roots} for short) of $\gg$.

\np
Let $\gh$ be a Cartan subalgebra of $\gg$ containing $\gt$ with corresponding root decomposition of $\gg$

\begin{equation} \label{eq12}
\gg = \gh \oplus (\oplus_{\alpha \in \Delta} \, \gg_\alpha),
\end{equation} 
where $\Delta$ and $\gg_\alpha$ are defined analogously to $R$ and $\gg_\nu$ above.  

\np
Roots systems and the corresponding decompositions \eqref{eq12} provide the cornerstone for developing the classification
and structure theory of simple complex Lie algebras.  The decompositions \eqref{krs} and \eqref{eq12} are closely related as
follows. The inclusion $\gt \subset \gh$ gives rise to a natural projection $\pi: \gh^* \longrightarrow \gt^*$ that relates \eqref{krs} to \eqref{eq12} as follows: 
\[\gm = \gh \oplus (\oplus_{\alpha \in \ker \pi} \, \gg_\alpha), \quad \quad R = \pi(\Delta) \backslash \{0\}, \quad \quad
\gg_\nu = \oplus_{\alpha \in \pi^{-1} (\nu)} \, \gg_\alpha.\]
It is thus natural to expect that \eqref{krs} and $R$ would inherit (at least some) properties from \eqref{eq12} and $\Delta$.

\np
In \cite{Ko} Kostant studied the decomposition  \eqref{krs} and the set $R$ 
in the case when $\gt = \gz(\gm)$. (See Remark \ref{rem-at}
for a discussion of the
condition $\gt = \gz(\gm)$.) In particular he proved that 
each $\gg_\nu$ is an irreducible $\gm$-module, that $[\gg_\mu, \gg_\nu] = \gg_{\mu+\nu}$ whenever $\mu, \nu, \mu+ \nu \in R$, 
that $R$ admits bases of simple roots, etc. 
In recognition of Kostant's work, we call $R$ the \textit{Kostant root system} (or the \textit{$\gt$-root system} for short) associated with the 
pair $(\gg, \gt)$  and the elements of $R$ -- \textit{Kostant roots} (or \textit{$\gt$-roots}). 

\np
In \cite{G} Greenstein studied Kostant roots systems for affine Lie algebras. He showed that $\gg_\nu$ is not necessarily an irreducible
$\gm$-module and proved that a sufficient condition of its irreducibility is $\dim \gg_\alpha = 1$ for every $\alpha \in \pi^{-1} (\nu)$.

\np
\point 
{\bf Reductive Lie superalgebras.}
The goal of this paper is to extend Kostant's results to Lie superalgebras. The simple finite dimensional Lie superalgebras fall into three
subclasses: basic classical, strange, and Cartan-type superalgebras. In terms of roots and root decompositions, the basic classical 
superalgebras are the closest analogs of simple finite dimensional Lie algebras. The basic classical superalgebras are the following: 
\[ \gsl(m|n), m\neq n, \quad \gpsl(m|m), \quad \gosp(m|2n), \quad \D(2,1;a), \quad \G(3), \quad \F(4).\]
Among the basic classical superalgebras, $\gpsl(m|m)$ are the only ones that are not  Kac-Moody superalgebras in the sense of Serganova, \cite{Se}.
A manifestation of this feature is the fact that the $\gpsl(2|2)$ has root spaces that are two dimensional. 
Moreover, the projection $\gsl(2|2) \longrightarrow \gpsl(2|2)$ identifies the roots and the corresponding root spaces of $\gsl(2|2)$ and $\gpsl(2|2)$
and, hence, $\gsl(2|2)$ also admits two dimensional root spaces. This problem is resolved only by enlarging $\gsl(2|2)$ to $\ggl(2|2)$.
Even though the root spaces of $\gsl(m|m)$ and $\gpsl(m|m)$ are one dimensional for $m \geq 3$, the $\gt$-root 
spaces $\gg_\nu$ are not necessarily irreducible $\gm$-modules when $\gg = \gpsl(m|m)$ even if $m > 2$, see Remark \ref{rem_psl}.
Thus, in our work, we study $\ggl(m|m)$ instead of $\gpsl(m|m)$.

\np
Below we identify the class of Lie superalgebras for which we develop the theory of Kostant roots.

\begin{definition} \label{reductive}
A finite dimensional Lie superalgebra $\gg$ is called a {\it reductive superalgebra} if 

\begin{equation} \label{eq1.11}
\gg = \ga \oplus \gg_1 \oplus \ldots \oplus \gg_l,
\end{equation}
where $\ga$ is an abelian Lie algebra and each $\gg_j$ for $1 \leq j \leq l$ is a simple Lie algebra or a 
superalgebra isomorphic to one of the following:
\[ \gsl(m|n), m\neq n, \quad \ggl(m|m), m \geq 1, \quad \gosp(m|2n), \quad \D(2,1;a), \quad \G(3), \quad \F(4).\] 
If  $\gg=\gg_{\e}\oplus \gg_{\o}$ is a reductive superalgebra, $\gt$ is a {\it toral subalgebra} of $\gg$ if 
$\gt \subset \gg_{\e}$ is an abelian algebra which acts semisimply on $\gg$. 
\end{definition} 

\np
Note that, unlike the roots systems of reductive Lie algebras, the roots
systems of reductive superalgebras admit bases which are not conjugate under the action of the Weyl group; this is a feature they share with
the Kostant root systems of simple Lie algebras.
 
\np
\point
{\bf Main results.}
Theorems \ref{main1} -- \ref{main3} below are the main results of this paper concerning the decomposition \eqref{krs} of a reductive Lie superalgebra
$\gg$ with respect to a toral subalgebra $\gt$ of $\gg$.
First we prove that the centralizer $\gm$ of $\gt$ itself 
is a reductive Lie superalgebra. 

\begin{theorem} \label{main1}
Let $\gg$ be a reductive Lie superalgebra, let $\gt$ be a toral subalgebra of $\gg$, and let
$\gg=\gm \, \oplus \, \left( \oplus_{\nu \in R} \, \gg_\nu\right)$ be the decomposition \eqref{krs}. Then 
$\gm$ is a reductive Lie superalgebra.
\end{theorem}

\np
The remaining two theorems establish properties of the set $R$ of $\gt$-roots. For these two theorems we work under the assumption that $\gt = \gz(\gg)$.

\begin{theorem} \label{main2} 
In the notation of Theorem \ref{main1} above, assume that $\gt = \gz(\gm)$.
If $\Sigma$ is a base of $\Delta$, then $\pi(\Sigma) \backslash \{0\}$  is a base of $R$ if and only if $\Sigma$ contains a base of $\Delta_\gm$, 
the root system of $\gm$. Moreover, every base $S$ of $R$ is of the form $\pi(\Sigma) \backslash \{0\}$
for some base $\Sigma$ of $\Delta$ containing a base of $\Delta_\gm$. In particular, the Kostant root system $R$ admits a base. 
\end{theorem}

\begin{theorem} \label{main3} 
Under the assumptions of Theorem \ref{main2} above we have
\begin{enumerate}
\item[(i)] for any $\nu \in R$, the $\gm$-module $\gg_\nu$ is irreducible;
\item[(ii)] if $\mu$, $\nu \in R$ and $\mu+\nu \in R$, then 
$[\gg_\mu, \gg_\nu]=\gg_{\mu+\nu}$;
\item[(iii)] if $\mu$, $\nu \in R$, and $\mu + k \nu \in R$, where $k \in \ZZ_{>0}$, then $\mu + j \nu \in R$ for every $0 \leq j \leq k$. 
\end{enumerate}
\end{theorem}

\np
\point
{\bf Approach.}
Theorems \ref{main1} -- \ref{main3} above are the analogs of results in \cite{Ko}. 

\np
Theorem \ref{main1} is relatively straightforward but
one needs to take special care to ensure that in $\gm$ every component of $[\gm,\gm]$ isomorphic to $\gs\gl(m|m)$ can be completed to
a component isomorphic to $\gg\gl(m|m)$.  

\np
Theorem \ref{main2} relies on the results of \cite{DFG} to characterize the bases of $R$.

\np
Theorem \ref{main3} is the most difficult of the three. 
Kostant's approach to studying the properties of the set $R$ and the decomposition \eqref{krs} parallels the classical approach to root systems and root
decompositions and makes a heavy use of a form on $\gt$ inherited from the restriction on $\gh$ of the Killing form of $\gg$. 
This approach does not extent to the case when $\gg$ is a superalgebra because a non-degenerate invariant form on $\gg$ is, in general,
not positive-definite. In order to overcome this problem we introduce a graph $\Gamma_{\Delta, \Sigma}$ associated with a root system $\Delta$ and a base 
$\Sigma$ of $\Delta$
and translate properties (i) -- (iii) into properties of $\Gamma_{\Delta, \Sigma}$. We then prove that the graphs $\Gamma_{\Delta, \Sigma}$ behave ``functorially'' with respect 
to projections like $\pi$. In the case when $\gg$ is a non-exceptional basic classical Lie superalgebra, part (i) of Theorem \ref{main3} then follows from the observation
that the corresponding graphs $\Gamma_{\Delta, \Sigma}$ are the same as the graphs associated with appropriate Kostant root systems of Lie algebras. 
In the case when $\gg$ is an exceptional Lie superalgebra, part (i) of Theorem \ref{main3} is a direct calculation. In order to provide a proof of Theorem \ref{main3}
independent of \cite{Ko} we also prove part (i) in the case when $\gm \subset \gg_\e$ using a result of Stembridge, \cite{St}.
Once we prove part (i) of Theorem \ref{main3}, the remaining statements are established using the functorial properties of the graphs $\Gamma_{\Delta, \Sigma}$. 

\np
\point 
{\bf Structure and contents.}
The paper is organized as follows. In Section \ref{reductive} we discuss the general properties of reductive Lie superalgebras and prove  
Theorem \ref{main1}. In Section \ref{bases} we prove Theorem \ref{main2}.
In Section \ref{graphs} we introduce and study the graphs $\Gamma_{\Delta, \Sigma}$.
In Section \ref{irr-sec} we prove Theorem \ref{main3}.
Section \ref{sec_hermitian} applies Theorem \ref{main2} to studying Hermitian symmetric pairs of Lie superalgebras.
Explicit calculations that complete the proof of part (i) of Theorem \ref{main3} for $\gg = \D(2,1;a), \G(3)$, and $\F(4)$ are provided in the Appendix.

\np
\point
{\bf Notation and conventions.}
\begin{itemize}
\item[-] The base field is $\CC$ and, unless explicitly stated otherwise, all vector spaces, Lie algebras and superalgebras, etc. are defined over $\CC$.
\item[-] The sets of positive (non-negative, etc.) integers are denoted by $\ZZ_{>0}$ ($\ZZ_{\geq 0}$, etc.). 
\item[-] For a subset $X$ of a vector space $V$, the span of $X$ is denoted by $\langle X \rangle$. 
\item[-] The elements of $\ZZ/2\ZZ$ are denoted  by $\e$ and $\o$; if $V$ is vector superspace (i.e, a 
$\ZZ/2\ZZ$-graded vector space), its graded components are denoted by $V_\e$ and $V_\o$.
\item[-] If $\Gamma$ is a graph, $v(\Gamma)$ and $e(\Gamma)$ denote the vertices and the edges of $\Gamma$.
\item[-] If $\gl$ is a Lie superalgebra, $\gz(\gl)$ denotes the center of $\gl$; if $\gk$ is a subalgebra of $\gl$, $\gc(\gk)$ denotes the
centralizer of $\gk$ in $\gl$.
\end{itemize}

\section{Reductive superalgebras} \label{reductive}

\np
Throughout the paper $\gg$ denotes a reductive Lie superalgebra  with a decomposition \eqref{eq1.11}. Note that, unlike the case when $\gg$ is a reductive Lie algebra,
it is not true that $\ga$ is the center of $\gg$. Instead, we have the following relationship.

\begin{lemma} \label{lemma200} Let $\gg$ be a reductive Lie superalgebra and  let $\gh$ be a Cartan subalgebra of $\gg$. Then
\[
\dim \gh = \rk \gg + \dim \gz(\gg),
\]
where $\rk \gg$ denotes the dimension of $\langle \Delta \rangle \subset \gh^*$, the vector space spanned by the roots of $\gg$.
\end{lemma}

\begin{proof} 
Notice that $\gh$ decomposes as $\gh = \ga \oplus \gh_1 \oplus \ldots \oplus \gh_l$, where $\gh_j$ is a Cartan subalgebra of $\gg_l$. 
Consider $\gg_j$. If $\gg_j \cong \ggl(m|m)$ for some $m$, then $ \dim \gz(\gg) =1$ and $\rk \gg_j = \dim \gh_j -1$;
if, on the other hand, $\gg_j$ is not isomorphic to $\ggl(m|m)$, then $\gz(\gg) = 0$ and $\rk \gg_j = \dim \gh_j$.
Taking into account the above, the decomposition 
\[
\gz(\gg)  = \ga \oplus \gz(\gg_1) \oplus \ldots \gz(\gg_l)
\]
completes the proof.
\end{proof}

\np
Let $\gt$ be a toral subalgebra of $\gg$. 
Fixing a Cartan subalgebra $\gh$ with $\gt\subset \gh$, we denote by $\pi$ the natural projection $\pi: \gh^* \longrightarrow \gt^*$. 
It is convenient to introduce the sets
$\Deo:=\Delta \cup \{0\}$ and $\Ro:= R \cup \{0\}$, where  $\Delta$ and $R$ are respectively the $\gh$-roots and the $\gt$-roots of $\gg$.
Note that $\pi(\Delta) = \Ro$.

\np Our first goal is to prove Theorem \ref{main1}. We start by establishing a result which is of independent interest.
Recall that a subset $\Sigma \subset \Delta$ is a 
\textit{base} of  $\Delta$ if $\Sigma$ is linearly independent and any element of $\Delta$ is
 an integral  linear combination of elements in $\Sigma$ with all coefficients in $\ZZ_{\geq 0}$ or  $\ZZ_{\leq 0}$. 

\begin{lemma} \label{lemma2.2} Let $\gg$ be a reductive Lie superalgebra with roots $\Delta$.  
If $W$ is a subspace of $\gh^*$, then there exists a base $\Sigma$ of $\Delta$ such that every element of $\Delta \cap W$ is a
combination of elements of the set $\Sigma_W := \Sigma \cap W$. 
\end{lemma}

\begin{proof} It suffices to prove the statement in the case when $\gg$ is a Kac-Moody Lie superalgebra, i.e.,  one of the superalgebras
$\gsl(m|n), m\neq n,  \ggl(m|m), m \geq 1, \gosp(m|2n)$,  $\D(2,1;a)$,  $\G(3)$, or $\F(4)$. Fix a linear function $f \in (\gh^*)^*$ which 
takes real values on $\Delta$ and such that, 
for every $\alpha \in \Delta$, we have $\alpha \in W$ if and only if $\alpha \in \ker f$. The set $P := \{\alpha \in \Delta \, | \, f (\alpha) \geq 0\}$
is a parabolic subset of $\Delta$, see \cite{DFG}. Proposition 2.10 in \cite{DFG} implies that there exists a base $\Sigma$ of $\Delta$ such that every positive
root with respect to $\Sigma$ belongs to $P$. In particular, $f$ takes only non-negative values on the elements of $\Sigma$. Now let $\alpha \in \Delta \cap W$. Then 
\[
\alpha = \sum_j c_j \beta_j + \sum_k d_k \gamma_k,
\]
where $\beta_j \in \Sigma_W$, $\gamma_k \in \Sigma \backslash \Sigma_W$ and all coefficients $c_j, d_k$ are non-negative or non-positive. Applying $f$ we obtain
\[ 0 = \sum_k d_k f(\gamma_k),\]
which shows that, unless $d_k = 0$ for every $k$, there exist $k_1$ and $k_2$ such that one of $f(\gamma_{k_1})$ and $f(\gamma_{k_2})$ 
positive and the other one is negative. This contradicts the fact that $f$ takes only non-negative values on the elements of $\Sigma$. Hence $d_k = 0$
for every $k$, i.e., $\alpha$ is a linear combination of elements of $\Sigma_W$.
\end{proof}

\np
A base $\Sigma$ with the above property is called a {\it $W$-adapted base of $\Delta$}.

\begin{proof}[\bf{Proof of Theorem \ref{main1}}]
As in the proof of Lemma \ref{lemma2.2} above it suffices to prove the statement in the case when $\gg$ is a Kac-Moody superalgebra. Since 
$\gm = \gh \oplus (\oplus_{\alpha \in \ker \pi} \gg_\alpha)$, Lemma \ref{lemma2.2} applies. Fix a $\ker \pi$-adapted base $\Sigma$ of $\Delta$ and 
set $\Sigma_\gm := \Sigma \cap \ker \pi$. Consider the Dynkin diagram associated with $\Sigma$ and let 
\[
\Sigma_\gm = \Sigma_1 \sqcup \Sigma_2 \sqcup \dots \sqcup \Sigma_l
\]
be the decomposition of $\Sigma_\gm$ into connected components. Denote the superalgebra generated by $\oplus_{\alpha \in \Sigma_k} (\gg_\alpha \oplus \gg_{-\alpha})$ 
by $\gg_k'$. Note that $\gg_k'$ is either a Kac-Moody superalgebra or is isomorphic to $\gsl(m_k|m_k)$ with $m_k\geq 1$. 

\np
We need to show that we can ``complete'' the subalgebras $\gg_k'$ to Kac-Moody superalgebras $\gg_k$, i.e.,
that there exist Kac-Moody superalgebras $\gg_k$ such that $\gg_k' \subset \gg_k$ and the sum $\gg_1 + \gg_2 + \dots + \gg_l$ is direct.
First we set $\gg_k := \gg_k'$ unless $\gg_k'$ is isomorphic to $\gsl(m_k|m_k)$. In order to deal with the subalgebras $\gg_k'$ 
 isomorphic to $\gsl(m_k|m_k)$, we consider two cases for $\gg$.

\np
If $\gg$ is isomorphic to $\gsl(m|n), m\neq n,  \ggl(m|m), m \geq 1$, or  $\gosp(m|2n)$, the roots of $\gg$ are expressed in terms of linear 
functions $\vep_i, \delta_j \in \gh^*$ with dual elements $h_i \in \gh$, see \cite{Ka}. For each $k$ we then define 
\[
\gg_k := \gg_k' + (\oplus_i \, \CC h_i), 
\]
where the sum is over all indices $i$ that appear in the expression in terms of $\vep$'s and $\delta$'s of a root in $\Sigma_k$. 
This definition agrees with the previous definition of $\gg_k$ when $\gg_k'$ is not isomorphic to $\gsl(m_k|m_k)$ and
$\gg_k \cong \ggl(m_k|m_k)$ when $\gg_k \cong \gsl(m_k|m_k)$. Moreover, the sum $\gg_1 + \gg_2 + \dots + \gg_l$ is direct.

\np
If $\gg$ is isomorphic to $\D(2,1;a)$,  $\G(3)$, or $\F(4)$, we notice, by inspecting the list of possible Dynkin diagrams,
that $\gg_k'$ is isomorphic to $\gsl(m_k|m_k)$ for at most one index $k$ and, for this $k$, $m_k = 1$, see the Appendix. 
If there is no such $k$ we are done. Assume now that $\gg_1' \cong \gsl(1|1)$. Let $h \in \gh$ be such that $[h, \gg_k] = 0$ for $k>1$
but $[h, \gg_1'] \neq 0$. Then $\gg_1 := \CC h \oplus \gg_1' \cong \ggl(1|1)$ and the sum $\gg_1 + \gg_2 + \dots + \gg_l$
is direct.

\np 
To complete the proof, we set $\ga:= \gc_\gm(\gg_1 \oplus \gg_2 \oplus \dots \oplus \gg_l)$. The decomposition
\[
\gm  = \ga \oplus \gg_1 \oplus \gg_2 \oplus \dots \oplus \gg_l
\]
shows that $\gm$ is a reductive Lie superalgebra.
\end{proof}

\np
We complete this section by recording some properties of the roots of reductive superalgebras that will be used in the rest of the paper.

\begin{proposition} \label{prop2.4} Let $\gg$ be a reductive Lie superalgebra with roots $\Delta$.  Then
\begin{enumerate}
\item[(i)] $-\Delta = \Delta$.
\item[(ii)] If $\alpha, \beta \in \Delta$, then $\alpha + \beta \in \Delta \cup \{0\}$ if and only if $[\gg_\alpha, \gg_\beta] \neq 0$.
\item[(iii)] If $\Sigma$ is a base of $\Delta$ and $\gamma$ is a positive root with respect to $\Sigma$,
then there exist positive roots $\gamma=\gamma_1,\gamma_2, \ldots, \gamma_N, \gamma_{N+1} = 0 \in \Deo$ such that
$\gamma_i - \gamma_{i+1} \in \Sigma$ for every $1 \leq i \leq N$.
\end{enumerate}
\end{proposition}

\begin{proof}
Statements (i) and (ii) are just rephrasings of the analogous statements of Proposition 2.5.5 in \cite{Ka}. Statement (iii)  follows from the
fact that, if $\gamma \not \in \Sigma$ is 
a positive root, then there exists $\alpha \in \Sigma$ such that $\gamma - \alpha$ is a (positive) root. To prove this fact, assume to the contrary 
that $\gamma - \alpha \not \in \Delta$  
for any $\alpha \in \Sigma$.  This implies that $\gg_\gamma$ is a lowest 
weight space of a proper submodule of the component of the adjoint representation of $\gg$ containing $\gg_\gamma$.
This contradiction completes the proof.
\end{proof}

\section{Bases} \label{bases}
 
\np In this section we prove Theorem \ref{main2} and establish some result about bases of root systems that will be needed in 
the rest of the paper.
The definition of a base of $\Delta$ extends in a natural way to subsets $S$ of $R$.

\begin{definition} A subset  $S \subset R$ is  a
\textit{base} of  $R$  if $S$ is linearly independent and any element of $R$ is
 an integral  linear combination of elements of $S$ with all
coefficients in $\ZZ_{\geq 0}$ or  $\ZZ_{\leq 0}$. The cardinality of $S$ is called {\it rank of $R$}.
\end{definition}

\np
We are now ready to prove Theorem \ref{main2}.

\begin{proof}[\bf{Proof of Theorem \ref{main2}}] We start with the observation 
that $\ker \pi = \langle \Delta_\gm \rangle$. Indeed, for any $\alpha \in \Delta_\gm$ and every $t \in \gt$, we have
$\alpha(t) = 0$ because $\gt = \gz(\gm)$, proving $\langle \Delta_\gm \rangle \subset \ker \pi$. Moreover, 
Lemma \ref{lemma200} applied to $\gm$ implies that
\[\dim \ker \pi = \dim \gh^* - \dim \gt^* = \dim \gh - \dim \gt = \rk \gm = \dim \langle \Delta_\gm\rangle,\]
proving that $\ker \pi = \langle \Delta_\gm \rangle$.

\np
Let $\Sigma = \{\alpha_1, \ldots, \alpha_k, \beta_1, \ldots, \beta_l\}$ be a base of $\Delta$, where $\alpha_1, \ldots, \alpha_k \in \Delta_\gm$ and 
$\beta_1, \ldots, \beta_l \in \Delta \backslash \Delta_\gm$. Then $\pi(\Sigma) \backslash \{0\} = \{\pi(\beta_1), \ldots, \pi(\beta_l)\}$. 
It is clear that every element
of $R$ is an integral linear combination of elements of $\pi(\Sigma) \backslash \{0\}$ with all coefficients in $\ZZ_{\geq 0}$ or $\ZZ_{\leq 0}$.
Hence $\pi(\Sigma) \backslash \{0\}$ is a base of $R$ if and only if $\pi(\beta_1), \ldots, \pi(\beta_l)$ are linearly independent. The latter itself is 
equivalent to $\langle \beta_1, \ldots, \beta_l\rangle \cap \ker \pi = 0$. Since $\ker \pi = \langle \Delta_\gm \rangle$, we decude that
$\pi(\Sigma) \backslash \{0\}$ is a base of $R$ if and only if $\langle \beta_1, \ldots, \beta_l\rangle \cap \langle \Delta_\gm \rangle = 0$ and
the last condition is equivalent to $\langle \alpha_1, \ldots, \alpha_k \rangle = \langle \Delta_\gm \rangle$. Since $\alpha_1, \ldots, \alpha_k$
are linearly independent, $\langle \alpha_1, \ldots, \alpha_k \rangle = \langle \Delta_\gm \rangle$ is equivalent to $\{\alpha_1, \ldots, \alpha_k\}$
being a base of $\Delta_\gm$.

\np
Next we prove that every base of $R$ is of the form  $\pi(\Sigma) \backslash \{0\}$. Let $S$ be a base of $R$ and consider the set

\begin{equation} \label{eq3.2}
P := \{ \alpha \in \Delta \, | \, \pi(\alpha) {\text { is a combination of elements of }} S {\text { with coefficients in }} \ZZ_{\geq 0} \}.
\end{equation}
It is a parabolic subset of $\Delta$. Lemma \ref{lemma2.2} and Proposition 2.10 in \cite{DFG} imply that there exists a base 
$\Sigma = \{\alpha_1, \ldots, \alpha_k, \beta_1, \ldots, \beta_l\}$ such that $\{\alpha_1, \ldots, \alpha_k\}$ is a basis of 
$\Delta_\gm$ and $\{\beta_1, \ldots, \beta_l\}$ is a subset of  $P$. Then both $S$ and $\{\pi(\beta_1), \ldots, \pi(\beta_l)\}$ are bases of $R$
which generate the same cone in the $\QQ$-vector space spanned by $R$. Hence $S = \{\pi(\beta_1), \ldots, \pi(\beta_l)\} = \pi(\Sigma) \backslash \{0\}$.

\np
Finally, the existence of bases of $R$ follows immediately from the existence of bases of $\Delta$ containing bases of $\Delta_\gm$ which follows 
from Lemma \ref{lemma2.2}.
\end{proof}

\np
Bases of $R$ are closely related to positive systems in $R$. A subset $R^+ \subset R$ is called {\it a positive system} if 
\[ {\text{(i) }} R = R^+ \cup R^-, \quad {\text{(ii) }} R^+ \cap R^- = \emptyset, \quad {\text {(iii) }} \nu_1, \nu_2 \in R^+, \nu_1 + \nu_2 \in R \implies \nu_1 + \nu_2 \in R^+.\]
An element $\nu \in R^+$ is called {\it indecomposable} if it cannot be written as $\nu = \nu_1 + \nu_2$ with $\nu_1, \nu_2 \in R^+$.

\begin{proposition} \label{simple1} 
If $S$ is a base of $R$, then the subset of $R$ consisting of $\ZZ_{\geq 0}$-combinations of $S$ 
is a positive system whose indecomposable elements are the elements of $S$. 
Conversely, if $R^+$ is a positive system in $R$, then the set of indecomposable elements is a base of $R$.
\end{proposition}

\begin{proof}
The first statement is obvious. To prove the converse, notice that
\[P := \{ \alpha \in \Delta \, | \, \pi(\alpha) \in R^+ \}\]
is a parabolic subset of $\Delta$ and then proceed as in the paragraph following \eqref{eq3.2} above.
\end{proof}

\np 
\begin{corollary} \label{cor3.5}
Let $\nu \in R$ be a primitive root, i.e., $k \nu \in R$ implies $|k| \geq 1$. Then there exists a base $S$ of $R$ containing $\nu$. 
\end{corollary}

\begin{proof}
Consider a hyperplane $\cH$ in the real vector space spanned by $R$ which contains no elements of $R$ and such that $\pm \nu$ are the elements of $R$
closest to $\cH$. Let $R^+$ be the set of elements of $R$ on the same side of $\cH$ as $\nu$. Then $R^+$ is a positive system in $R$
and $\nu \in R$ is indecomposable. By Proposition \ref{simple1}, $\nu$ belongs to the base of $R^+$.
\end{proof}

\np
It is convenient to introduce some additional notation in which a functorial property of the projection $\Delta \longrightarrow \Ro$ can be expressed naturally.
Let $I$ be a base of $\Delta_\gm$ and assume that $I \neq \emptyset$, i.e., that $\gt \neq \gh$. 
Then $\gt^*\cong \gh^*/\langle I \rangle$ and $\Ro =\pi(\Delta)$.  Slightly abusing notation we denote $\Ro$ by
$\Delta/\langle I \rangle$. Notice that the notation $\Delta/\langle I \rangle$ carries more information than just the Kostant root system $R$, namely the
base $I$ of $\Delta_\gm$. We will use the notation $\Delta/\langle I \rangle$ without an explicit reference to $R$ or $\gm$. We also denote
the natural projection $\Delta \longrightarrow \Delta/\langle I \rangle$ by $\pi_I$. In fact, 
$\Delta/\langle I \rangle$ and $\pi_I$ make sense as long as $I$ is a subset of a base of $\Delta$. In what follows, whenever using the notation $\Delta/\langle I \rangle$
(or $\pi_I$) we will implicitly assume that $I$ is contained in a base of $\Delta$. Finally, we will also use the self-explanatory notation
$R/\langle I\rangle$ and $\pi_I: R \longrightarrow R/\langle I\rangle$.

\np
As a direct consequence of Theorem \ref{main2} we have the following statement.

\begin{proposition} \label{prop3.5}
The projection $\gh^* \mapsto \gh^*/\langle I \rangle$ induces a bijection
\[\{ {\text {\rm {bases of }}} \Delta {\text {\rm { containing }}} I \} \quad \longleftrightarrow \quad \{ {\text {\rm {bases of }}} \Delta/\langle I \rangle\}. \]
\vskip-2em \qed
\end{proposition}

\np 
Next we establish a functorial property of  projections of root systems which will play a central role in the rest of the paper.

\begin{proposition} \label{simple3}
Consider Kostant root systems $R_1$ and $R_2$ with corresponding projections
$\pi_1: \Delta \longrightarrow \Ro_1$ and $\pi_2: \Delta \longrightarrow \Ro_2$. 
Then the following are equivalent:
\begin{enumerate}
\item[(i)] There exists a surjection $\pi_{12}: \Ro_1\longrightarrow \Ro_2$
such that $\pi_2=\pi_{12} \circ \pi_1$; 
\item[(ii)] there exist a base $\Sigma$ of $\Delta$ and nested subsets $I_1$ and $I_2$ of $\Sigma$, i.e. $I_1 \subset I_2 \subset \Sigma$  
such that $\Ro_1 = \Delta/\langle I_1 \rangle$ and $\Ro_2 = \Delta/\langle I_2 \rangle$.
\end{enumerate}
Assuming that {\rm(i)} and {\rm(ii)} hold,  $S_2 \subset R_2$ 
is a base of $R_2$ if and only if $S_2=\pi_{12}(S_1)\setminus\{0\}$ for some base 
$S_1$ of $R_1$ containing $\pi_1(I_2 \setminus I_1)$.
\end{proposition}

\begin{proof}
The fact that (ii) implies (i) is clear. To prove that (i) implies (ii), assume that $\pi_{12}$ exists. Then 
$\Delta_{\gm_1} \subset \Delta_{\gm_2}$, where $\gm_1$ and $\gm_2$ are the reductive subalgebras of $\gg$ corresponding to $R_1$ and $R_2$ respectively.
We can then choose $I_1 \subset I_2$ such that $\Delta_{\gm_j} = \langle I_j \rangle$ for $j = 1,2$.
The existence a base $\Sigma$ containing $I_2$ (and hence  $I_1 \subset I_2 \subset \Sigma$) follows from Theorem \ref{main2}.

\np
If $S_2$ is a base of $R_2$, then $S_2=\pi_2(\Sigma)\setminus  \{0\}$ for some base $\Sigma$ of $\Delta$ containing  $I_2$.  Since
$I_1 \subset I_2$, then $S_1:= \pi_1(\Sigma)\setminus\{0\}$ is a base of $R_1$ containing $\pi_1(I_2\setminus I_1)$.
Conversely, if $S_1$ is a base of $R_1$ containing $\pi_1(I_2\setminus I_1)$
and $S_2=\pi_{12}(S_1)\setminus\{0\}$, then $S_1=\pi_1(\Sigma)\setminus\{0\}$ for some base $\Sigma$ containing $I_2$ (hence,
$\Sigma \supset I_2 \supset I_1$).
Then, applying Theorem \ref{main2} and the equivalence of (i) and (ii), we conclude that 
$\pi_2(\Sigma)\setminus\{0\}=\pi_{12}(\pi_1(\Sigma))\setminus\{0\}
=\pi_{12}(S_1)\setminus\{0\}=:S_2$ is a base of $R_2$.
\end{proof}

\np
Note that  Proposition \ref{simple3} can be generalized by replacing $\Delta$ with any Kostant root
system $R$. However we will not need a statement of this generality, so we omit this discussion. Moreover,
the constructions and results so far can be carried over to infinite dimensional Kac-Moody Lie superalgebras. 
Since the modifications are minimal, we leave the details to the reader.

\section{Graphs of Kostant root systems} \label{graphs}

\np
In this section we want to associate various graphs to a Kostant root
system $R$ and relate their properties to the
properties of the decomposition (\ref{krs}).

\begin{definition}
Let $R$ be a Kostant root system with a base $S$. 
The graph $\Gamma_{R,S}$ is defined as follows:
the vertices  of  $\Gamma_{R,S}$ are the elements of $R$ and 
there is an edge between two vertices $\mu$ and
$\nu$ if and only if $\mu-\nu \in \pm S$. We say 
that the edge between $\mu$ and $\nu$ is \textit{labelled} by
 the corresponding simple root $\mu-\nu$ or $\nu - \mu$. 
  \end{definition}

\begin{definition}
In the notation above, if $I \subset S$,  $\pi_I: R \longrightarrow R/\langle I \rangle$, and 
$\tau \in   R/\langle I \rangle$, the graph $\Gamma_{R,S,I}^\tau \subset \Gamma_{R,S}$ is defined as follows:
 $v(\Gamma_{R,S,I}^\tau) = \pi^{-1}(\tau)$ and there is an edge between two vertices $\mu$ and
$\nu$ if and only if $\mu -\nu \in \pm I$. 
\end{definition} 

\np 
Clearly the graph $\Gamma_{\Delta, \Sigma, I}^\nu$, where $I \subset \Sigma$ are bases of $\Delta_\gm$ and $\Delta$ respectively,
 is related to the Kostant root space $\gg_\nu$. More precisely we have the following statement.

\begin{proposition} \label{conn-graph2} Let $\gg$ be a reductive Lie superalgebra with decomposition \eqref{krs}.
Assume that $\Sigma$ is a base of $\Delta$ containing a base $I$ of $\Delta_\gm$.  
If  $\nu \in R$,
then $\gg_\nu$ is an irreducible $\gm$-module if and only if the graph $\Gamma^\nu_{\Delta, \Sigma, I}$ is connected. 
\end{proposition}

\begin{proof} Note that $\Gamma^\nu_{\Delta, \Sigma, I}$ is the set of weights of the $\gm$-module $\gg_\nu$. 
If $\Gamma^\nu_{\Delta, \Sigma, I}$ is connected then every weight space of $\gg_\nu$ generates all of its weight spaces and
since any submodule of $\gg_\nu$ is a weight submodule, we conclude that $\gg_\nu$ is irreducible.
Conversely, if $\Gamma^\nu_{\Delta, \Sigma, I}$ is not connected, every connected component contains a highest weight and hence $\gg_\nu$
is not irreducible.
\end{proof}

\begin{definition} \label{def4.40}
Let $R$ be a Kostant root system with a base $S$ and let $I$ be a subset of $S$ with corresponding projection
$\pi_I$. Set $S/\langle I \rangle := \pi_I(S) \backslash \{0\}$. 
If $\Gamma$ is a subgraph of $\Gamma_{R,S}$ such that $v(\Gamma) \cap \langle I \rangle = \emptyset$,
we define the subgraph $\pi(\Gamma)$ of $\Gamma_{R/\langle I \rangle, S/\langle I \rangle}$ as follows:
$v(\pi(\Gamma)) = \pi(v(\Gamma))$ and $e(\pi(\Gamma)) = \pi(e(\Gamma)) \backslash \{\text{loops}\}$,
where a loop is an edge whose vertices coincide.
\end{definition}

\np
The following statement is obvious.

\begin{proposition} \label{graph-conn} In the notation of Definition \ref{def4.40},
if $\Gamma$ is connected, so is $\pi(\Gamma)$. \qed
\end{proposition}

\np
We now want to make few observations on graphs that will be instrumental
for proving Theorem \ref{main2} and the irreducibility of the $\gm$-module $\gg_\nu$.
Let $\Sigma$ be a base of $\Delta$ and assume that $I \subset J \subset \Sigma$ . 
By Proposition \ref{simple3} we have the commutative diagram

\begin{equation} \label{eq460}
\xymatrix{
         &  \Delta \ar[dr]^{\pi_J} \ar[dl]_{\pi_I}  & \\
\Delta/\langle I \rangle   \ar[rr]^{\pi_{I,J}} &  &  \Delta/\langle J \rangle}
\end{equation}
where $\pi_I$, $\pi_J$, and $\pi_{I,J}$ are the projections described by 
Proposition \ref{simple3}.

\np
For a nonzero element $\nu \in  \Delta/\langle J\rangle$ consider the graphs
$\Gamma^\nu_{\Delta, \Sigma, J} \subset \Gamma_{\Delta,\Sigma}$ and 
$\Gamma^\nu_{\Delta/\langle I \rangle, \Sigma/\langle I \rangle, J/\langle I \rangle} \subset \Gamma_{\Delta/\langle I \rangle,\Sigma/\langle I \rangle}$.
The lemma below compares $\pi_I(\Gamma^\nu_{\Delta, \Sigma, J})$ and 
$\Gamma^\nu_{\Delta/\langle I \rangle, \Sigma/\langle I \rangle, J/\langle I \rangle}$
as subgraphs of $\Gamma_{\Delta/\langle I \rangle,\Sigma/\langle I \rangle}$.

\begin{lemma} \label{graph-lemma}
In the notation above we have
\begin{enumerate}
\item[(i)] $v(\pi_I(\Gamma^\nu_{\Delta, \Sigma, J}))=v(\Gamma^\nu_{\Delta/\langle I \rangle, \Sigma/\langle I \rangle, J/\langle I \rangle})$;
\item[(ii)]  $e(\pi_I(\Gamma^\nu_{\Delta, \Sigma, J})) \subset e(\Gamma^\nu_{\Delta/\langle I \rangle, \Sigma/\langle I \rangle, J/\langle I \rangle})$;
\item[(iii)] if $\Gamma^\tau_{\Delta, \Sigma, K}$ is connected for every $K \subset \Sigma$ and every nonzero $\tau \in \Delta/\langle K \rangle$, then
$e(\pi_I(\Gamma^\nu_{\Delta, \Sigma, J})) = e(\Gamma^\nu_{\Delta/\langle I \rangle, \Sigma/\langle I \rangle, J/\langle I \rangle})$.
\end{enumerate}
\end{lemma}

\begin{proof}
Statements (i) and (ii) follow from the commutativity of \eqref{eq460}.

\np
To prove (iii) we first consider the case when the cardinality $J \setminus I$ equals 1. 
Let $J = I \cup \{\alpha\}$. Set $\lambda := \pi_I(\alpha) \in \Sigma/\langle I \rangle$ and fix 
$\mu \in v(\Gamma^\nu_{\Delta/\langle I \rangle, \Sigma/\langle I \rangle, J/\langle I \rangle})$. Note that $J/\langle I \rangle = \{\lambda\}$.
Then $v(\Gamma^\nu_{\Delta/\langle I \rangle, \Sigma/\langle I \rangle, J/\langle I \rangle}) \subset \mu + \ZZ \lambda$, say 
\[v(\Gamma^\nu_{\Delta/\langle I \rangle, \Sigma/\langle I \rangle, J/\langle I \rangle}) = \{\mu + a \lambda \, | \, a \in A\}\]
for some finite set $A \subset \ZZ$. Moreover, two vertices $\mu + a' \lambda$ and $\mu + a'' \lambda$ are connected if and only if $|a' - a''| = 1$.
Furthermore,

\begin{equation} \label{eq480}
v(\Gamma^\nu_{\Delta, \Sigma, J}) = \sqcup_{a \in A}  \, 
v(\Gamma^{\mu+a \lambda}_{\Delta, \Sigma, I}).
\end{equation}
The edges of $\Gamma^\nu_{\Delta, \Sigma, J}$ are labeled by elements of $J = I \cup \{\alpha\}$; more precisely, elements of $I$ label
edges within components $\Gamma^{\mu+a \lambda}_{\Delta, \Sigma, I}$ and $\alpha$ labels edges between different components in \eqref{eq480}.
Since  $\Gamma^\nu_{\Delta, \Sigma, J}$ is connected,
$A$ is an interval and any two consecutive components in \eqref{eq480} are connected by $\alpha$. This proves that every edge of 
$\Gamma^\nu_{\Delta/\langle I \rangle, \Sigma/\langle I \rangle, J/\langle I \rangle}$ lifts to an edge of $\Gamma^\nu_{\Delta, \Sigma, J}$, i.e., 
that $e(\pi_I(\Gamma^\nu_{\Delta, \Sigma, J})) = e(\Gamma^\nu_{\Delta/\langle I \rangle, \Sigma/\langle I \rangle, J/\langle I \rangle})$.

\np
To complete the proof for general $I \subset J \subset \Sigma$, consider an edge $e$ of 
$\Gamma^\nu_{\Delta/\langle I \rangle, \Sigma/\langle I \rangle, J/\langle I \rangle}$ between the vertices $\mu'$ and $\mu''$ labeled by 
$\lambda \in  J/\langle I \rangle$. Let $\lambda = \pi_I(\alpha)$ with $\alpha \in J$. Set $K: = I \cup \{\alpha\}$ and $\tau:= \pi_{I,K}(\mu') = \pi_{I,K}(\mu'')$.
We have already proved that  
$e(\pi_I(\Gamma^\tau_{\Delta, \Sigma, K})) = e(\Gamma^\tau_{\Delta/\langle I \rangle, \Sigma/\langle I \rangle, K/\langle I \rangle})$.
In particular $e$ lifts to an edge of $\Gamma^\tau_{\Delta, \Sigma, K}$ labeled by $\alpha$. Since $v(\Gamma^\tau_{\Delta, \Sigma, K}) \subset
v(\Gamma^\nu_{\Delta, \Sigma, J})$ and $\alpha \in J$, we conclude that $e$ lifts to an edge of $\Gamma^\nu_{\Delta, \Sigma, J}$.
\end{proof}

\np
We complete this section by recording a result  which is crucial for the proof of 
 Theorem \ref{main3} (i). 

\np
\begin{proposition} \label{conn-graph3}
In the notation of Lemma \ref{graph-lemma}, if $\Gamma^\nu_{\Delta, \Sigma, J}$ is connected, so is 
$\Gamma^\nu_{\Delta/\langle I \rangle, \Sigma/\langle I \rangle, J/\langle I \rangle}$.
\end{proposition}

\begin{proof} Assume that $\Gamma^\nu_{\Delta, \Sigma, J}$ is connected.
Proposition \ref{graph-conn} implies that $\pi_I(\Gamma^\nu_{\Delta, \Sigma, J})$ is connected.
Since, by Lemma \ref{graph-lemma} (i) and (ii), $\Gamma^\nu_{\Delta/\langle I \rangle, \Sigma/\langle I \rangle, J/\langle I \rangle}$
has the same vertices as $\pi_I(\Gamma^\nu_{\Delta, \Sigma, J})$ and, possibly, more edges, we conclude that 
$\Gamma^\nu_{\Delta/\langle I \rangle, \Sigma/\langle I \rangle, J/\langle I \rangle}$ is connected as well.
\end{proof}

\section{Proof of Theorem \ref{main3}} \label{irr-sec}

\np 
We start with proving part (i) of Theorem \ref{main3}  in the case when $\gm \subset \gg_\e$. 
The proof is based on a result of Stembridge, see \cite{St}, Section 1. Let $\Lambda$ denote the
weight lattice of a semisimple Lie algebra $\gs$ and let $Q \subset \Lambda$ denote the root lattice of $\gs$.
According to Corollary 1.13 in \cite{St}, each nontrivial coset of $\Lambda/Q$ contains exactly one
minuscule weight. Further, Remark 1.11 and Proposition 1.12 in \cite{St} combine into the following statement.

\begin{proposition} \label{stem-prop}
Let $\Lambda/Q=\{\omega_1 = [0], \omega_2, \dots, \omega_l\}$.
Then there exist $\lambda_1, \dots, \lambda_l$, with
$\lambda_i \in \omega_i$, such that any simple module $W$
of $\gs$, whose support is in $\omega_i$, has $\lambda_i$ as
one of its weights. \qed
\end{proposition}

\begin{proposition} \label{main-even}
Theorem \ref{main3} (i) holds under the additional assumption that $\gm \subset \gg_\e$.
\end{proposition}

\begin{proof}
Let $\gs=[\gm,\gm]$. Then $\gs$ is a semisimple Lie algebra,
$\gg_\nu$ is an $\gs$-module which is irreducible if and only if 
$\gg_\nu$ is an irreducible $\gm$-module, since $\gt$
acts semisimply on it and $\gm=\gs\oplus \gt$. 
Fix  a Cartan
subalgebra $\gh_\gs$ of $\gs$. 
Notice that all the weights of $\gg_\nu$, with respect
to $\gh_\gs$, belong to the same coset
of $\Lambda/Q$ and that the weight spaces of  $\gg_\nu$ are all
one dimensional. Moreover, $\gg_\nu$ is a semisimple $\gs$-module.
Assuming that it is not irreducible, we would conclude that it contains (at least) two
simple submodules $W'$ and $W''$.
Then, by Proposition \ref{stem-prop}, there would be $\lambda_i$ which is a weight space of both $W'$ and $W''$.
This contradicts
the fact that all the weight spaces of $\gg_\nu$ are one dimensional.
\end{proof}

\begin{remark} 
In the case when $\gg$ is a Lie algebra, Proposition \ref{main-even} provides a proof of Theorem \ref{main3} (i) which is different 
from Kostant's proof. Moreover, the proof above also establishes the analogous result of Greenstein, \cite{G}.
\end{remark}

\begin{remark} 
Notice that the spaces $\gg_\nu$ are pairwise non-isomorphic 
as $\gm$-modules, but not necessarily as $\gs$-modules. The simplest 
example is provided by the decomposition of the adjoint representation of $\gs\gl_3$ 
as a module over a subalgebra of $\gs\gl_3$ isomorphic to $\gs\gl_2$ and generated by 
root spaces.
\end{remark}

\begin{remark} \label{rem-at}
The hypothesis $\gt=\gz(\gm)$ is essential for the irreducibility of the $\gm$-modules $\gg_\nu$.  
A simple example when $\gg_\nu$ fail to be irreducible if $\gt \neq \gz(\gm)$ is provided by
taking $\gg = \gs\gl_3$ and $\gt = \CC \, t$, where $t = (E_{11} - E_{33})$. Then $\gm = \gh$, the Cartan subalgebra of
$\gg$ consisting of diagonal matrices and $\gt \neq \gz(\gm) = \gm$. Moreover, decomposition \eqref{krs} in this case becomes
$\gg = \gg_{-2} \oplus \gg_{-1} \oplus \gm \oplus \gg_1 \oplus \gg_2,$
where both $\gg_{-1}$ and $\gg_1$ are  two-dimensional and hence reducible $\gm$-modules.
\end{remark}

\np
We are now ready to prove Theorem \ref{main3}.

\begin{proof}[{\bf Proof of Theorem \ref{main3}}] It is sufficient to prove the theorem when $\gg$ is a simple Lie algebra or is
isomorphic to one of the superalgebras
$\gsl(m|n)$ with $m\neq n$, $\ggl(m|m)$ for $m \geq 1$, $\gosp(m|2n)$, $\D(2,1;a)$, $\G(3)$, or $\F(4)$ and throughout the 
proof we will work under this assumption.

\np
{\bf (i)} Proposition \ref{main-even} deals with the case when $\gg$ is a simple Lie algebra.

\np Next we consider the case when $\gg$ is isomorphic to one of the superalgebras 
$\gsl(m|n)$ with $m\neq n$, $\ggl(m|m)$ for $m \geq 1$, or $\gosp(m|2n)$. The proof in this case 
follows from the observation that $\Deo$ is isomorphic to the Kostant root system $\tilde{\Delta}/\langle I \rangle$ for 
an appropriate simple Lie algebra $\tilde{\gg}$ and the funtoriality of the graphs associated with root systems.
Here are the details. 

\np
We say that two Kostant root systems $R'$ and $R''$ are isomorphic if there is an isomorphism $\varphi: V' \longrightarrow V''$, where 
$V'$ and $V''$ are the vector spaces spanned by $R'$ and $R''$ respectively, which restricts to a bijection between $R'$ and $R''$.
Comparing the roots of the Lie superalgebras, \cite{Ka}, with the Kostant roots of classical simple Lie algebras, \cite{DR},
we see that $\Deo$ is isomorphic to $\tilde{\Delta}/\langle I \rangle$, where
\begin{enumerate}
\item[(a)] if $\gg \cong \gsl(m|n)$ or $\gg \cong \ggl(m|m)$, then $\Delta$ is isomorphic to $\tilde{\Delta}$ with $\tilde{\gg} = \gsl_{m+n}$ or 
$\tilde{\gg} = \gsl_{2m}$ respectively;
\item[(b)] if $\gg \cong \gosp(2m|2n)$, then $\Delta$ is isomorphic to $\tilde{\Delta}/\langle I \rangle$ with $\tilde{\gg} = \gso_{2N}$ for appropriate $N$ and $I$
(this case is referred as ``type D, case I'' in \cite{DR});
\item[(c)] if $\gg \cong \gosp(2m+1|2n)$, then $\Delta$ is isomorphic to $\tilde{\Delta}/\langle I \rangle$ with $\tilde{\gg} = \gso_{2N+1}$ for appropriate $N$ and $I$
(this case is referred as ``type B, case I'' in \cite{DR}).
\end{enumerate}
Note that the choices for $\tilde{\Delta}$ and $I$ above are not unique.

\np 
Let now that $R$ be a Kostant root system of $\gg$ and assume that $\Deo \cong \tilde{\Delta}/\langle I \rangle$ as above. 
Then Proposition \ref{simple3} implies that there exist a base $\tilde{\Sigma}$ of $\tilde{\Delta}$ and subsets $I \subset J \subset \tilde{\Sigma}$ 
for which diagram \eqref{eq460} takes the form

\[
\xymatrix{
         &  \tilde{\Delta} \ar[dr]^{\pi_J} \ar[dl]_{\pi_I}  & \\
\Deo = \Delta/\langle I \rangle   \ar[rr]^{\pi_{I,J}} &  &  {\Delta/\langle J \rangle} = \Ro\,.}
\]
Let $\nu \in R$. Then $\tilde{\gg}_\nu$ is irreducible and, by Proposition \ref{conn-graph2}, the graph $\Gamma_{\tilde{\Delta}, \tilde{\Sigma}, J}^\nu$
is connected. Hence, by Proposition  \ref{conn-graph3}, the graph  $\Gamma_{\Delta, \tilde{\Sigma}/\langle I \rangle, J/I}^\nu$ is connected too. Finally, applying 
Proposition \ref{conn-graph2} again we conclude that $\gg_\nu$ is irreducible thus completing the proof in this case. 

\np
When $\gg$ is isomorphic to one of the superalgebras $\D(2,1;a)$, $\G(3)$, or $\F(4)$, the irreducibility of
$\gg_\nu$ is a direct verification which, with a little bit of care, can be carried out by hand. We provide the details of
this calculation in the Appendix.

\np
{\bf(ii)} Let $\mu, \nu$ and $\mu + \nu \in R$. Since $\gg_{\mu + \nu}$ is an irreducible $\gm$-module, it suffices to prove that
$[\gg_\mu, \gg_\nu] \neq 0$.

\np
First we consider the case when  $\mu$  is primitive, i.e., $k \mu \in R$ implies $|k| \geq 1$. 
By Corollary \ref{cor3.5} we may choose a base $S$ of $R$ such that $\mu \in S$. Fix a base $\Sigma$ of $\Delta$ and $I \subset \Sigma$
such that $\Ro = \Delta/\langle I \rangle$ and $S = \Sigma/\langle I \rangle$. Let $\alpha \in \Delta$ be such that $\pi_I(\alpha) = \mu$ and
let $J := I \cup \{\alpha\}$. Lemma \ref{graph-lemma} applies and we conclude that any edge of $\Gamma_{R, S, I}^\tau$
lifts to an edge of $\Gamma_{\Delta, \Sigma, J}^\tau$, where $\tau := \pi_{IJ}(\nu) = \pi_{IJ}(\mu+\nu)$. The assumption that
$\mu, \nu$, and $\mu + \nu \in R$ implies that $\nu$ and $\mu + \nu$ are vertices of $\Gamma_{R, S, I}^\tau$ connected by an edge labeled by $\mu$.
This edge lifts to an edge labeled by $\alpha$ connecting two vertices of $\Gamma_{\Delta, \Sigma, J}^\tau$ -- one in 
$\Gamma_{\Delta, \Sigma, I}^\nu$ and another one in $\Gamma_{\Delta, \Sigma, I}^{\mu + \nu}$. Say, 
$\beta \in \Gamma_{\Delta, \Sigma, I}^\nu$ and $\gamma \in \Gamma_{\Delta, \Sigma, I}^{\mu + \nu}$. In other words, we have $\alpha, \beta, \gamma \in \Delta$
such that $\alpha + \beta= \gamma$ and $\pi_I(\alpha) = \mu, \pi_I(\beta) = \nu$, and $\pi_I(\gamma) = \mu+ \nu$. Since, by Proposition \ref{prop2.4} (ii),
$[\gg_\alpha, \gg_\beta] = \gg_{\gamma} \neq 0$, we conclude that
\[
[\gg_\mu, \gg_\nu] \supset [\gg_\alpha, \gg_\beta] = \gg_{\gamma} \neq 0,
\]
which completes the proof in this case.

\np
The case when $\nu$ is primitive is treated in the same way. Before we proceed to the case when neither $\mu$ nor $\nu$ is primitive, we need to introduce
some notation.

\np
Let $\Sigma$ be a base of $\Delta$
and let $\alpha \in \Sigma$. For $s \in \ZZ$, consider the set
\[\Delta[\alpha,s] := \{ \beta \in \Delta \, | \, {\text{ the coefficient of }} \alpha {\text { in the }} \Sigma- {\text{decomposition of }} \beta {\text { is divisible by }} s\}.\]
The axiomatic description of the root systems of Lie superalgebras in \cite{Se1} implies that $\Delta[\alpha,s]$ is a root system itself.
If $\Ro = \Delta/\langle I \rangle$ is a Kostant root system, set $R[\alpha, s] := R \cap \Delta[\alpha,s]$.
Notice that, if  $\alpha \not \in I$, then $I$ is contained in a base of $\Delta[\alpha,s]$ because every element of $\Delta$ 
which is in the span of $I$ belongs to $\Delta[\alpha,s]$. Moreover, $\Ro[\alpha, s] = \Delta[\alpha, s]/\langle I \rangle$.

\np
Assume now that $\mu = s_1 \mu'$ and $\nu = s_2 \nu'$, where $\mu'$ and $\nu'$ are primitive. Fix a base $\Sigma$ of $\Delta$ and $I \subset \Sigma$
such that $\Ro = \Delta/\langle I \rangle$ and $\alpha, \beta, \gamma \in \Delta$ with 
$\pi_I(\alpha) = \mu, \pi_I(\beta) = \nu, \pi_I(\gamma) = \mu+ \nu$. We may assume that both $\mu$ and $\nu$
are positive with respect to $\Sigma/\langle I \rangle$ and hence $s_1, s_2 \in \ZZ_{>0}$.
The assumption $\mu = s_1 \mu'$ and $\nu = s_2 \nu'$ implies that, for every $\theta \in \Sigma \setminus I$, 
$s_1$ and $s_2$ divide the respective coefficients $c_1^\theta$ and $c_2^\theta$ of $\theta$ in
the decompositions of $\alpha$ and $\beta$ with respect to the base $\Sigma$. Assume first that, for some $\theta \in  \Sigma \setminus I$, 
the coefficients $c_1^\theta$ and $c_2^\theta$ have a common divisor $d >0$ and
consider $\Delta[\theta, d]$. Its cardinality is smaller than that of $\Delta$ and $\mu, \nu, \mu+\nu \in R[\theta,d]$.
Hence an induction on the cardinality of $\Delta$, using the result for  $\mu, \nu, \mu+\nu \in R[\theta,d]$, 
implies that  $[\gg_\mu, \gg_\nu] \neq 0$. Assume now that, for every $\theta \in  \Sigma \setminus I$, 
$c_1^\theta$ and $c_2^\theta$ are coprime. This implies in particular that $c_1^\theta \geq s_1$, $c_2^\theta \geq s_2$, and $s_1$ and $s_2$ are coprime.
Then, for every $\theta \in  \Sigma \setminus I$, we have $c_3^\theta = c_1^\theta + c_2^\theta \geq s_1 + s_2 \geq 5$, where $c_3^\theta$ is the coefficient
of $\theta$ in the decompositions of $\gamma$ with respect to the base $\Sigma$. 

\np
A quick inspection of the highest roots of root systems shows that
only the root system of $\E_8$ allows coefficients greater than 4. In the labeling of \cite{B}, the highest root of $\E_8$ is
\[2 \alpha_1 + 3 \alpha_2 + 4 \alpha_3 + 6 \alpha_4 + 5 \alpha_5 + 4 \alpha_6 + 3 \alpha_7 + 2 \alpha_8\]
and, moreover, there is no root in which both $\alpha_4$ and $\alpha_5$ enter with coefficient 5. Hence, the only cases we have not reduced to a smaller
root system yet is when $I = \Sigma \setminus \{\alpha_4\}$ or $I = \Sigma \setminus \{\alpha_5\}$ and $\mu = 2 \mu'$ and $\nu = 3 \mu'$. In both cases 
it is easy to find roots $\alpha \in \pi_I^{-1}(\mu)$, $\beta \in \pi_I^{-1}(\nu)$ such that $\alpha + \beta$ is a root. This completes the proof.

\np
{\bf (iii)} Without lost of generality we may assume that $\nu$ is primitive. Fix a base 
$\Sigma$ of $\Delta$, $I\subset \Sigma$ such that $\Ro = \Delta/\langle I \rangle$, and
$\alpha \in \Sigma \setminus I$ with $\pi_I(\alpha) = \nu$. We have to consider separately two cases for $\mu$.

\np
Assume first that $\mu$ is not proportional to $\nu$. Set $J := I \cup \{\alpha\}$ and $\tau:= \pi_{IJ}(\mu) = \pi_{IJ}(\mu + j \nu)$, see diagram \eqref{eq460}. 
By part (i) above, $\Gamma_{\Delta, \Sigma, J}^\tau$ is connected and so is $\Gamma_{R, \Sigma/\langle I \rangle, \Sigma/\langle I\rangle}^\tau$
 by Proposition  \ref{conn-graph3}. Since the vertices of the latter graph are all elements of $R$ of the form $\mu + j \nu$ and the edges are labeled by $\nu$,
 we conclude that $\mu, \mu+ k \nu \in R$ implies that $\mu + j \nu$ for every $j$ between $0$ and $k$.
 
 \np
If $\mu$ is proportional to $\nu$, it is sufficient to prove that $k \nu \in R$ with $k \in \ZZ_{>0}$ implies $j \nu \in R$ for every $0 \leq j \leq k$.
Let $\gamma \in \pi_I^{-1}(k \nu)$. Then $\gamma$ is in the span of $J$ and the coefficient of $\alpha$ in the decomposition of $\gamma$ with respect to
$J \subset \Sigma$ equals $k$. Consider a sequence $\gamma = \gamma_1, \gamma_2, \ldots, \gamma_N$ as in Proposition \ref{prop2.4} (iii). 
Then the projection under $\pi_I$ of the set $\{\gamma_j \, | \, 1 \leq j \leq N+1\}$ is exactly the set $\{j \nu \, | \, 0 \leq j \leq k\}$.
\end{proof}

\np
We complete this section with a remark justifying why we consider $\ggl(m|m)$ instead of $\gpsl(m|m)$.

\begin{remark} \label{rem_psl} If $\gg = \gsl(m|m)$, then $\gg_\nu$ may be a reducible $\gm$-module.
Here is an example for $m \geq 3$ in the usual presentation of $\gsl(m|m)$ as $2m \times 2m$-matrices. 
Let $\gt$ be the span of $H_1:= E_{11}-E_{22}, H_2:= E_{33} + \ldots + E_{2m-3, 2m-3}$, and $H_3:=E_{2m-1, 2m-1} - E_{2m, 2m}$. 
Then $\gm = \gh + \gg'$, where $\gg' \cong \gsl(m-2|m-2)$ is the subalgebra of traceless $(m-2) \times (m-2)$-matrices indexed by $3 \leq i , j \leq 2m-3$.
It is immediate to verify that $\gt = \gz(\gm)$. Define $\nu \in \gt^*$ by $\nu(H_1) = 1$, $\nu(H_2) = 0$, and $\nu(H_3) = -1$. Then $\gg_\nu$ is spanned by 
$E_{1,2m-1}$ and $E_{2m, 2}$ and each of these vectors spans an $\gm$-submodule of $\gg_\nu$. We leave it to the reader to check that, after projecting,
we obtain an example for $\gg = \gpsl(m|m)$ as well.
\end{remark}

\section{Hermitian symmetric pairs and admissible systems} \label{sec_hermitian}

\np
In this section we provide an application of the results of section \ref{bases} to Hermitian symmetric pairs. 
If $\gg$ is a simple Lie algebra and $(\gg, \gk)$ is a non-compact
Hermitian symmetric pair, then $\gg$ admits a decomposition

\begin{equation} \label{eq3.7}
\gg=\gk \oplus \gp=\gk \oplus \gp^+ \oplus \gp^-,
\end{equation}
where $\gk$ is a reductive subalgebra of $\gg$ with one dimensional center $\gt$, $\gp^\pm$ are abelian subalgebras of $\gg$ which are also
irreducible $\gk$-modules. (For a detailed discussion on Hermitian symmetric pairs, see \cite{H-C}.) It is clear that $\gk = \gc(\gt)$ and $\gt = \gz(\gk)$.  Moreover, the $\gt$-roots of $\gg$ are $R = \{\pm \nu\}$ with $\gg^{\pm \nu} = \gp^\pm$. In other words,
the decomposition \eqref{eq3.7} is just a particular case of \eqref{krs}. By Theorem \ref{main2}, $R$ is the projection of a base $S$ of $\Delta$
and $\nu = \pi(\alpha)$ for some $\alpha \in S$. The fact that $R = \{\pm \nu\}$ implies that the coefficient of $\alpha$ in the decomposition of any positive root
of $\gg$ into a combination of simple roots must equal either $0$ or $1$. This condition is equivalent to the condition that 
the coefficient of $\alpha$ in the decomposition the highest root of $\gg$  equals $1$, providing the list of all non-compact Hermitian symmetric pairs $(\gg, \gk)$ 
in the case when $\gg$ is a simple Lie algebra. The compact Hermitian symmetric pairs are just the pairs $(\gg, \gg)$. 

\np 
In \cite{CFV} Carmeli, Fioresi, and Varadarajan 
introduced a generalization of Hermitian symmetric pairs for basic classical 
Lie superalgebras $\gg \not \cong \gpsl(m|m)$
(see also \cite{CF1}). More precisely, $(\gg, \gk)$ is a {\it Hermitian symmetric pair} corresponding to the Cartan decomposition

\begin{equation} \label{eq3.8} \gg = \gk \oplus \gp, \end{equation}
if the following conditions are satisfied:
\begin{enumerate}
\item[(i)] $\gk = \gk_\e \subset \gg_\e$ contains a Cartan subalgebra of $\gg$;
\item[(ii)] for every simple ideal $\gs$ of $\gg_\e$, $(\gs, \gk \cap \gs)$ is a Hermitian symmetric pair;
\item[(iii)] $\gp_\o = \gg_\o$. 
\end{enumerate}
The roots $\Delta_\gk$ of $\gk$ and $\Delta_\gp$ of $\gp$ are called respectively {\it compact} and {\it non-compact}.
A positive system $\Delta^+ \subset \Delta$ is {\it admissible} if 
\[ (\Delta_\gk + \Delta_\gp^+) \cap \Delta \subset \Delta_\gp^+ \quad {\text {and}} \quad (\Delta_\gp^+ + \Delta_\gp^+) \cap \Delta \subset \Delta_\gp^+.\]
The following proposition appears as Theorem 3.3 in \cite{CFV}; we provide a proof based on Theorem \ref{main2}.

\begin{proposition} \label{prop 3.9}
Every admissible positive system $\Delta_\e^+$ in $\Delta_\e$ can be extended to an admissible positive system $\Delta^+ \subset \Delta$.
\end{proposition}

\begin{proof} The crucial observation is that $\langle \Delta_\gk \rangle \cap \Delta = \Delta_\gk$. Note that, since $\gk$ is a reductive subalgebra of $\gg_\e$,
$\langle \Delta_\gk \rangle \cap \Delta_\e = \Delta_\gk$ is automatic. However, the fact that $\langle \Delta_\gk \rangle$ contains no odd roots follows from
the properties (i) and (ii) of the decomposition \eqref{eq3.8} and is not true for a general reductive subalgebra $\gk$ of $\gg_\e$. The proof of this property is a straightforward 
verification for each $\gk$, see \cite{CFV} for the list of all possible $\gk$.

\np 
Let $S_\e$ be the base of $\Delta_\e^+$. Since $\Delta_\e^+$ is admissible, $I := S_\e \cap \Delta_\gk$ is a base of $\Delta_\gk$ and hence 
$\langle \Delta_\gk \rangle = \langle I \rangle$. Let $S_R$ be a base of the Kostant system $\Ro := \Delta/\langle I \rangle$. 
Then $S_R$ is a projection of a base $S$ of $\Delta$ containing $I$ and the positive system $\Delta^+$ with base $S$ is admissible and contains $\Delta_\e^+$.
\end{proof}

\np
The root system $R$ above (and all of its bases) can be described explicitly for each Hermitian symmetric pair $(\gg, \gk)$. As a consequence, it is easy to 
obtain quickly all admissible positive systems $\Delta^+ \subset \Delta$ extending a given admissible system $\Delta_\e^+ \subset \Delta_\e$. Indeed, it is sufficient
to provide a list of all positive systems $R^+ \subset R$ extending $R_\e^+ \subset R_\e$, where $R_\e := \pi(\Delta_\e) \backslash \{0\}$ and 
$R_\o := \pi(\Delta_\o)$.
Below we provide the complete list of Kostant root systems $R$ and their positive systems $R^+$ corresponding to Hermitian symmetric pairs $(\gg,\gk)$, 
cf. the proof of Theorem 3.3 in \cite{CFV}. In most cases we use natural parametrizations of the positive systems $R^+$ and $R_\e^+$ by the elements of 
symmetric or dihedral groups of small order; depending on the particular context, we denote a group of order 2 by $\cS_2$ or $\cZ_2$. 

\point 
If $\gg = \gsl(m|n)$ with $m\neq n$ and $(\gg,\gk)$ is a  Hermitian symmetric pair, 
then $\gk$ admits a real form isomorphic to $\gsu(p,q) \oplus \gsu(r,s) \oplus \gu(1)$ where $p+q = m, r+s = n$
and we allow $p,q,r,$ or $s$ to take value zero. We may (and will) assume that $0 \leq p \leq q \leq m, 0 \leq r \leq s \leq n$. 

\np
If both $p$ and $q$ are nonzero, then 
$\gk = (\gsl_p \oplus \gsl_q \oplus \CC) \oplus  (\gsl_r \oplus \gsl_s \oplus \CC) \oplus \CC$ and 
\[R_\e = \{\pm(\bar{\vep}_1 - \bar{\vep}_2), \pm(\bar{\delta}_1 - \bar{\delta}_2)\}, \quad R_\o = \{ \pm(\bar{\vep}_i - \bar{\delta}_j) \, | \, 1 \leq i, j \leq 2\}.\]
The positive systems $R^+$ are indexed by elements of the symmetric group $\cS_4$ permuting 
$\bar{\vep}_1,  \bar{\vep}_2, \bar{\delta}_1 , \bar{\delta}_2$
 and the positive systems $R_\e^+$ are indexed by elements of
$\cS_2 \times \cS_2 \subset \cS_4$. Hence the positive systems $R^+$ extending a given $R_\e^+$ are indexed by the cosets in $\cS_4/\cS_2 \times \cS_2$.
In particular, there are 6 such extensions.

\np
If $p=0$ but $q \neq 0$, then 
$\gk = \gsl_m  \oplus  (\gsl_r \oplus \gsl_s \oplus \CC) \oplus \CC$ and 
\[R_\e = \{\pm(\bar{\delta}_1 - \bar{\delta}_2)\}, \quad R_\o = \{ \pm(\bar{\vep} - \bar{\delta}_j) \, | \,  j = 1, 2\}.\]
The positive systems $R^+$ are indexed by elements of the symmetric group $\cS_3$ and the positive systems $R_\e^+$ are indexed by elements of
$\cS_2  \subset \cS_3$. Hence the positive systems $R^+$ extending a given $R_\e^+$ are indexed by the cosets in $\cS_3/\cS_2$.
In particular, there are 3 such extensions.

\np
The case when $p \neq 0$ but $q = 0$ is analogous to the case when $p=0$ but $q \neq 0$.

\np
If $p = q =0$, then 
$\gk = \gsl_m  \oplus  \gsl_n \oplus \CC$ and 
\[R_\e = \emptyset, \quad R_\o = \{ \pm(\bar{\vep} - \bar{\delta}) \}.\]
The positive systems $R^+$, and, as a tautology,  the positive systems 
$R^+$ extending a given $R_\e^+$, are indexed by elements of the symmetric group $\cS_2$. In particular, there are 2 such extensions.

\point 
If $\gg = \gosp(2m+1 | 2n)$, $m > 0$, and $(\gg,\gk)$ is a  Hermitian symmetric pair, 
then $\gk$ admits a real form isomorphic to $\gso_\RR(2, 2m-1) \oplus \gsp_n(\RR)$ and hence 
$\gk \simeq (\gsl_2 \oplus \gso_{2m-1} \oplus \CC) \oplus (\gsl_n \oplus \CC)$. Thus
\[R_\e = \{\pm \bar{\vep}, \pm 2\bar{\delta}\}, \quad R_\o = \{ \pm \bar{\vep} \pm \bar{\delta}, \pm \bar{\delta}\}.\]
The positive systems $R^+$ are indexed by elements of the dihedral group $\cD_4$ acting on the square with vertices $\{\pm \bar{\vep} \pm \bar{\delta}\}$
(or, equivalently, the square with vertices $\{\pm \bar{\vep}, \pm \bar{\delta}\}$).
The positive systems $R_\e^+$ are indexed by elements of
$\cZ_2 \times \cZ_2 \subset \cD_4$. Hence the positive systems $R^+$ extending a given $R_\e^+$ are indexed by the cosets in $\cD_4/\cZ_2 \times \cZ_2$.
In particular, there are 2 such extensions.

\point 
If $\gg = \gosp(1 | 2n)$ and $(\gg,\gk)$ is a  Hermitian symmetric pair, 
then $\gk$ admits a real form isomorphic to $\gsp_n(\RR)$ and hence 
$\gk \simeq \gsl_n \oplus \CC$. Thus
\[R_\e = \{ \pm 2\bar{\delta}\}, \quad R_\o = \{ \pm \bar{\delta}\}.\]
The positive systems $R^+$ are indexed by elements of the group $\cZ_2$ as are the positive systems $R_\e^+$. 
Hence there is a unique positive system $R^+$ extending a given $R_\e^+$.

\point
 If $\gg = \gosp(2 | 2n)$ and $(\gg,\gk)$ is a  Hermitian symmetric pair, 
then $\gk$ admits a real form isomorphic to $\gso_\RR(2) \oplus \gsp_n(\RR)$ or 
a real form isomorphic to $\gso_\RR(2) \oplus \gsp(n)$. In the former case 
$\gk \simeq \CC \oplus (\gsl_n \oplus \CC)$ and in the latter $\gk \simeq \CC \oplus \gsp_n$. 

\np 
When $\gk \simeq \CC \oplus (\gsl_n \oplus \CC)$, we have
\[R_\e = \{ \pm 2\bar{\delta}\}, \quad R_\o = \{ \pm \bar{\vep} \pm \bar{\delta}\}.\]
There are 6 positive systems $R^+$ and 2 positive systems $R_\e^+$.
Hence there are 3 positive systems $R^+$ extending a given $R_\e^+$.

\np 
When $\gk \simeq \CC \oplus \gsp_n$, we have
\[R_\e =\emptyset, \quad R_\o = \{ \pm \bar{\vep}\}.\]
The positive systems $R^+$, and, as a tautology,  the positive systems 
$R^+$ extending a given $R_\e^+$, are indexed by elements of the  group $\cZ_2$. In particular, there are 2 such extensions.

\point 
If $\gg = \gosp(2m | 2n)$, $m > 1$, and $(\gg,\gk)$ is a  Hermitian symmetric pair, 
then $\gk$ admits a real form isomorphic to $\gso_\RR(2, 2m-2) \oplus \gsp_n(\RR)$ or 
a real form isomorphic to $\gso^*(2m) \oplus \gsp(n)$. In the former case 
$\gk \simeq (\gso_{2m-2} \oplus\CC) \oplus (\gsl_n \oplus \CC)$ and in the latter $\gk \simeq (\gsl_m \oplus\CC) \oplus (\gsl_n \oplus \CC)$. 

\np 
When $\gk \simeq (\gso_{2m-2} \oplus\CC) \oplus (\gsl_n \oplus \CC)$, we have
\[R_\e = \{\pm \bar{\vep}, \pm 2\bar{\delta}\}, \quad R_\o = \{ \pm \bar{\vep} \pm \bar{\delta}, \pm \bar{\delta}\}\]
and when $\gk \simeq (\gsl_m \oplus\CC) \oplus (\gsl_n \oplus \CC)$, we have
\[R_\e = \{\pm 2 \bar{\vep}, \pm 2\bar{\delta}\}, \quad R_\o = \{ \pm \bar{\vep} \pm \bar{\delta}\}.\]
In both cases the positive systems $R^+$ are indexed by elements of the dihedral group $\cD_4$ and the positive systems $R_\e^+$ are indexed by elements of
$\cZ_2 \times \cZ_2 \subset \cD_4$. Hence the positive systems $R^+$ extending a given $R_\e^+$ are indexed by the cosets in $\cD_4/\cZ_2 \times \cZ_2$.
In particular, in each case there are 2 such extensions.

\point
If $\gg = \D(2,1;\alpha)$ and $(\gg,\gk)$ is a  Hermitian symmetric pair, 
then $\gk$ admits a real form isomorphic to $\gsl_2(\RR) \oplus \gsl_2(\RR) \oplus \gsl_2(\RR)$ or 
a real form isomorphic to $\gsl_2(\RR) \oplus \gsu(2) \oplus \gsu(2)$. In the former case 
$\gk \simeq \CC^3$ is a Cartan subalgebra of $\gg$ and in the latter $\gk \simeq \gsl_2 \oplus \gsl_2 \oplus \CC$. 

\np
When $\gk \simeq \CC^3$ is a Cartan subalgebra, the Kostant roots are just the usual roots
\[R_\e = \{\pm 2 \bar{\vep}_1, \pm 2 \bar{\vep}_2, \pm 2 \bar{\vep}_3\}, \quad R_\o = \{\pm  \bar{\vep}_1 \pm  \bar{\vep}_2 \pm  \bar{\vep}_3\}.\]
There are 32 positive systems $R^+$ and 8 positive systems $R_\e^+$. Hence there are 4 positive systems $R^+$ extending a given $R_\e^+$.

\np
When $\gk \simeq \gsl_2 \oplus \gsl_2 \oplus \CC$, we have 
\[R_\e = \{\pm 2 \bar{\vep}\}, \quad R_\o = \{\pm  \bar{\vep}\}.\]
Both the positive systems $R^+$ and $R_\e^+$ are indexed by $\cZ_2$ and every positive system $R_\e^+$ extends to a unique positive system $R^+$. 

\point
If $\gg = \F(4)$ and $(\gg,\gk)$ is a  Hermitian symmetric pair, 
then $\gk$ admits a real form isomorphic to $\gsu_2 \oplus \gso_\RR(2,5)$ or 
a real form isomorphic to $\gsl_2(\RR) \oplus \gso_\RR(7)$. In the former case 
$\gk \simeq \gsl_2 \oplus (\gso_5 \oplus \CC)$ and in the latter $\gk \simeq \CC \oplus \gso_7$. 

\np 
When $\gk \simeq \gsl_2 \oplus (\gso_5 \oplus \CC)$, we have
\[R_\e = \{\pm \bar{\vep}\}, \quad R_\o = \{ \pm \frac{1}{2}\bar{\vep}\}\]
and when $\gk \simeq \CC \oplus \gso_7$, we have
\[R_\e = \{\pm \bar{\delta}\}, \quad R_\o = \{ \pm \frac{1}{2}\bar{\delta}\}.\]
In both cases both the positive systems $R^+$ and $R_\e^+$ are indexed by $\cZ_2$ and every positive system $R_\e^+$ extends to a unique positive system $R^+$. 

\point
If $\gg = \G(3)$ and $(\gg,\gk)$ is a  Hermitian symmetric pair, 
then $\gk$ is isomorphic to $\gk \simeq  \G_2 \oplus \CC$. We have 
\[R_\e = \{\pm 2\bar{\delta}\}, \quad R_\o = \{ \pm \bar{\delta}\}.\]
Both the positive systems $R^+$ and $R_\e^+$ are indexed by $\cZ_2$ and every positive system $R_\e^+$ extends to a unique positive system $R^+$.

\np
Note finally that the discussion above in combination with Theorem \ref{main3} recovers Lemma 3.4 of \cite{CFV}.

\section*{Appendix. Irreducibility of $\gg_\nu$ for 
$\gg \cong \D(2,1;\alpha)$, $\F(4)$, and $\G(3)$} \label{app}

\np
The irreducibility of $\gg_\nu$  for the exceptional Lie superalgebras
amounts to a direct check. For completeness we list the bases
of $\D(2,1;\alpha)$, $\F(4)$, $\G(3)$
up to Weyl group equivalence ($W$-equivalence for short) and we list the routine checks 
that we leave to the reader. Taking into account
that the case of $\gm$ even is taken care by Proposition \ref{main-even},
our strategy is as follows:

\begin{itemize}
\item[Step 1:] Pick $I_1=\{\alpha\}$, $\alpha$ odd, in a base of $\Delta$
 and consider
$\pi_{I_1}:\Delta \longrightarrow R_1=\Delta/\langle I_1 \rangle$. Check that $\gg_\nu$
is irreducible for all $\nu \in R_1$.
Repeat this step for all possible choices of $I_1$.

\item[Step 2:] Enlarge $I_1$ to $I_2=\{\alpha,\beta\}$, where $\beta$ is chosen
in any possible way in the same base of $\Delta$ as $\alpha$.
Consider $\pi_{I_2}:\Delta \longrightarrow R_2=\Delta/\langle I_2 \rangle$. 
Check that $\gg_\nu$ is irreducible for all $\nu \in R_2$.
Repeat this step for all possible choices of $I_2$.

\item[Step 3:] If $\gg = \F(4)$, enlarge $I_2$ to $I_3=\{\alpha,\beta, \gamma\}$, where $\gamma$ is chosen
in any possible way in the same base of $\Delta$ as $I_2$.
Consider $\pi_{I_3}:\Delta \longrightarrow R_3=\Delta/\langle I_3 \rangle$. 
Check that $\gg_\nu$ is irreducible for all $\nu \in R_3$.
Repeat this step for all possible choices of $I_3$.
\end{itemize}

\np
Since the rank of exceptional Lie superalgebras is at most $4$ this
procedure will terminate after at most $3$ steps. Notice that our
choices can always be made up to $W$-equivalence.

\np
We now provide the list of all possible sets $I_1, I_2$ and $I_3$ as above, leaving the necessary
checks to the reader.

\np
{\bf 7.1.}
\underline{$\gg = \D(2,1;\alpha)$}. In this case $\Delta_\e=\{\pm 2\delta_i\}$, $\Delta_\o=
\{\pm \delta_1\pm \delta_2 \pm \delta_3\}$, $i,j=1,2,3$, and
$\gg_\e=A_1 \oplus A_1 \oplus A_1$.
Because of the Weyl group equivalence, signs are immaterial.
Notice the symmetry in the root expressions with respect to a permutation
of the indices (which is not a Weyl equivalence). 
 With this in mind, there is only one base to consider:
$$
S=\{\delta_1-\delta_2-\delta_3,\, 2\delta_2,\, 2\delta_3\}.
$$
The  only choice for $I_1$ is $I_1=\{\delta_1-\delta_2-\delta_3\}$. 
After renaming $\delta_i$, the only choice for $I_2$ is 
$I_2=\{\delta_1-\delta_2-\delta_3, 2\delta_1\}$.

\np
{\bf 7.2.} \underline{$\gg =\G(3)$}. In this case $\Delta_\e=\{\ep_i - \ep_j,\, \pm \ep_i,\, \pm 2\delta\}$,
$\Delta_\o=\{\pm \ep_i\pm \delta,\, \pm \delta)\}$, $i,j=1,2,3$,
$\ep_1+ \ep_2 + \ep_3=0$, and 
$\gg_\e=G_2 \oplus A_1$. 
There are four inequivalent bases:
$$
\begin{array}{rl}
S_1&= \left\{\ep_2,\, \ep_3-\ep_2,\, \delta+\ep_1 \right\}; \\
S_2&=\left\{ \delta-\ep_3,\, \ep_3-\ep_2,\, -\delta-\ep_1 \right\}; \\ 
S_3&=\left\{  -\delta+\ep_3,\, \delta-\ep_2,\, \ep_2 \right\};\\ 
S_4&= \left\{ \ep_3-\ep_2,\, -\delta+\ep_2,\, \delta \right\}.
\end{array}
$$
It is sufficient to consider $I_1=\{\delta+\ep_1\}$. As for $I_2$,
there are the following possibilities (with some redundancies):
$$
\{\delta+\ep_1,\, \delta \pm \ep_2\}, \quad \{\delta+\ep_1, \, \alpha\},
$$
where $\alpha=\delta$, $\ep_1$, $\ep_2$, $\ep_1\pm \ep_2$, $\ep_2\pm\ep_3$.

\np
{\bf 7.3.}  \underline{$\gg = \F(4)$}. In this case $\Delta_\e=\{\pm \ep_i \pm \ep_j, \pm \ep_i, \pm \delta\}$,
$\Delta_\o=\{(1/2)(\pm \ep_1\pm \ep_2 \pm \ep_3 \pm \delta)\}$, $i,j=1,2,3$, and
$\gg_\e=B_3 \oplus A_1$. 
There are six inequivalent bases:
$$
\begin{array}{rl}
S_1&= \left\{ \frac{1}{2}(\ep_1+\ep_2+\ep_3+\delta), 
\ep_1-\ep_2, \ep_2-\ep_3, -\ep_1  \right\};\\
S_2&=\left\{ -\frac{1}{2}(\ep_1+\ep_2+\ep_3+\delta), 
\ep_1-\ep_2, \ep_2-\ep_3, \frac{1}{2}(-\ep_1+\ep_2+\ep_3+\delta) \right\};  \\
S_3&=\left\{ -\ep_1, \frac{1}{2}(\ep_1-\ep_2+\ep_3+\delta),  
 \ep_2-\ep_3, -\frac{1}{2}(-\ep_1+\ep_2+\ep_3+\delta) \right\}; \\
S_4&= \left\{ \frac{1}{2}(-\ep_1-\ep_2+\ep_3+\delta),  
 -\frac{1}{2}(\ep_1-\ep_2+\ep_3+\delta),
\frac{1}{2}(\ep_1+\ep_2-\ep_3+\delta),  \ep_1-\ep_2\right\}; \\
S_5&= \left\{ \frac{1}{2}(\ep_1+\ep_2-\ep_3-\delta),-\ep_1 ,
\delta, \ep_1-\ep_2\right\};\\
S_6&= \left\{\delta, \ep_2-\ep_3,-\frac{1}{2}(\ep_1+\ep_2-\ep_3+\delta),
 \ep_1-\ep_2\right\}.
\end{array}
$$
Up to $W$-equivalence, there is only one choice for $I_1$:
$$
I_1=\{\frac{1}{2}(\ep_1+\ep_2+\ep_3+\delta)\}. \qquad
$$
Examining which bases contain $I_1$ we have
the following choices for $I_2$ consisting of two odd roots
(with some redundancies): 
$$
\begin{array}{rl}
I_2^1&=\{ \frac{1}{2}(\ep_1+\ep_2+\ep_3+\delta), \,
\frac{1}{2}(-\ep_1+\ep_2+\ep_3+\delta) \}; \\
I_2^2&=\{\frac{1}{2}(\ep_1+\ep_2+\ep_3+\delta), \,
\frac{1}{2}(\ep_1+\ep_2-\ep_3-\delta)\};\\
I_2^3&=\{\frac{1}{2}(\ep_1+\ep_2+\ep_3+\delta), \,
\frac{1}{2}(\ep_1-\ep_2-\ep_3-\delta)\}; \\
I_2^4&= \{\frac{1}{2}(\ep_1+\ep_2+\ep_3+\delta), \,
\frac{1}{2}(-\ep_1-\ep_2-\ep_3+\delta)\}; \\
I_2^5&=\{\frac{1}{2}(\ep_1+\ep_2+\ep_3+\delta), \,
\frac{1}{2}(\ep_1-\ep_2-\ep_3+\delta)\}. \\
\end{array}
$$
Since $W$ permutes the indices, 
$I_2^1$ is equivalent to $I_2^3$ and $I_2^2$ is equivalent to $I_2^5$.

\np
The choices for $I_2$ consisting of an even and an odd root are:
$$
\{\frac{1}{2}(\ep_1+\ep_2+\ep_3+\delta), \, \alpha \},
\qquad \alpha \in \{\delta, \ep_1, \ep_1 \pm \ep_2\}.
$$
Notice that the case when $\alpha=\ep_1$ is the same as $I_2^1$, 
 the case when $\alpha=\ep_1+\ep_2$ is the same as $I_2^2$, and the case when $\alpha=\delta$ is the same as  $I_2^2$,
so the only new case to verify is when $\alpha=\ep_1-\ep_2$.

\np
Finally, we determine the possible choices for $I_3$. There is only one set $I_3$ containing $3$ odd
roots --  the odd roots in $S_4$. The remaining possibilities
for $I_3$ are:
$$
I_2^i \cup \{\alpha\} \qquad  {\text{and}} \qquad
\{\frac{1}{2}(\ep_1+\ep_2+\ep_3+\delta), \alpha, \beta \},
$$
where $\alpha, \beta \in   \{\delta, \ep_1, \ep_1 \pm \ep_2 \}$.

\np
Verifying that the spaces $\gg_\nu$ are irreducible $\gm$-modules in the cases listed above is a tedious but straightforward calculation which we omit.

\end{document}